\documentclass[12pt]{article}
\title{Random Subnetworks of Random Sorting Networks}
\author{Omer Angel \and Alexander E. Holroyd}
\date{12 November 2009}

\usepackage{amsmath,amsthm,amssymb,amsfonts}
\usepackage{enumerate,graphics}
\usepackage[hypertex]{hyperref}

\newenvironment{mylist}{
\begin{enumerate}[(i)]\setlength{\itemsep}{0pt}\setlength{\parskip}{0pt}
  \setlength{\parsep}{0pt}}{\end{enumerate}}

\newtheorem{thm}{Theorem}
\newtheorem{lemma}[thm]{Lemma}
\newtheorem{cor}[thm]{Corollary}
\newtheorem{prop}[thm]{Proposition}
\newtheorem{conj}[thm]{Conjecture}

\renewcommand{\P}{\mathbb{P}}
\newcommand{\E}{\mathbb{E}}
\newcommand{\R}{\mathbb{R}}
\newcommand{\Z}{\mathbb{Z}}
\newcommand{\dof}{\bf\boldmath}
\newcommand{\h}{\tfrac12}
\newcommand{\ess}{S}
\newcommand{\id}{\textrm{id}}

\begin{document}
\maketitle
\renewcommand{\thefootnote}{}
\footnotetext{{\bf\hspace{-6mm}Key words:} sorting network,
random sorting, reduced word, Polya urn}
\footnotetext{{\bf\hspace{-6mm}2010 Mathematics Subject
Classifications:} 60C05, 05E10, 68P10}
\footnotetext{\hspace{-6mm}Funded in part by Microsoft Research
and NSERC.}
\renewcommand{\thefootnote}{\arabic{footnote}}

\begin{abstract}
  A sorting network is a shortest path from $12\cdots n$ to $n\cdots21$ in
  the Cayley graph of $S_n$ generated by nearest-neighbor swaps. For $m\leq
  n$, consider the random $m$-particle sorting network obtained by choosing
  an $n$-particle sorting network uniformly at random and then observing
  only the relative order of $m$ particles chosen uniformly at random. We
  prove that the expected number of swaps in location $j$ in the subnetwork
  does not depend on $n$, and we provide a formula for it. Our proof is
  probabilistic, and involves a Polya urn with non-integer numbers of
  balls. From the case $m=4$ we obtain a proof of a conjecture of
  Warrington. Our result is consistent with a conjectural limiting law of
  the subnetwork as $n\to\infty$ implied by the great circle conjecture
  Angel, Holroyd, Romik and Vir\'ag.
\end{abstract}

\section{Introduction}

Let $\ess_n$ be the symmetric group of all permutations
$\sigma=(\sigma(1),\ldots,\sigma(n))$ on $\{1,\ldots,n\}$, with composition
given by $(\sigma\tau)(i):=\sigma(\tau(i))$. For $1\leq s \leq n-1$ denote
the adjacent transposition or {\dof swap} at location $s$ by $\tau_s :=
(\,s\ \ s+1\,)=(1,2,\ldots,s+1,s,\ldots,n)\in\ess_n$. Denote the {\dof
  identity} $\id:=(1,2,\ldots,n)$ and the {\bf reverse} permutation
$\rho:=(n,\ldots,2,1)$. An $n$-particle {\dof sorting network} is a
sequence $\omega=(s_1,\ldots,s_N)$, where $N:=\binom{n}{2}$, such that
\[
\tau_{s_1} \tau_{s_2} \cdots \tau_{s_N}=\rho.
\]
For $1\leq t\leq N$ we refer to $s_t=s_t(\omega)$ as the $t$th {\dof swap
  location}, and we call the permutation
$\sigma_t=\sigma_t(\omega):=\tau_{s_1} \cdots \tau_{s_t}$ the {\dof
  configuration at time $t$}. We call $\sigma_t^{-1}(i)$ the {\dof location
  of particle $i$} at time $t$.

Given an $n$-particle sorting network $\omega$ and a subset $A$ of
$\{1,\ldots,n\}$ of size $m$, the induced {\dof subnetwork} $\omega|_A$ is
the $m$-particle sorting network obtained by restricting attention to the
particles in $A$. More precisely, if the elements of $A$ are
$a_1<a_2<\cdots<a_m$, delete from each configuration $\sigma_t$ of $\omega$
all elements not in $A$, and replace $a_i$ with $i$, to give a permutation
in $\ess_m$, then remove all duplicates from the resulting sequence of
permutations; the result is the sequence of configurations of $\omega|_A$.
See Figure \ref{subnet}.

\begin{figure}
\centering
\begin{picture}(0,0)(0,0)
\put(0,0){$1$}\put(0,20){$2$}\put(0,40){$3$}\put(0,60){$4$}\put(0,80){$5$}
\put(207,0){$5$}\put(207,20){$4$}\put(207,40){$3$}\put(207,60){$2$}\put(207,80){$1$}
\put(270,20){$1$}\put(270,40){$2$}\put(270,60){$3$}
\put(340,20){$3$}\put(340,40){$2$}\put(340,60){$1$}
\put(230,37){\Huge$\rightarrow$}
\end{picture}
\resizebox{!}{3.1cm}{\includegraphics{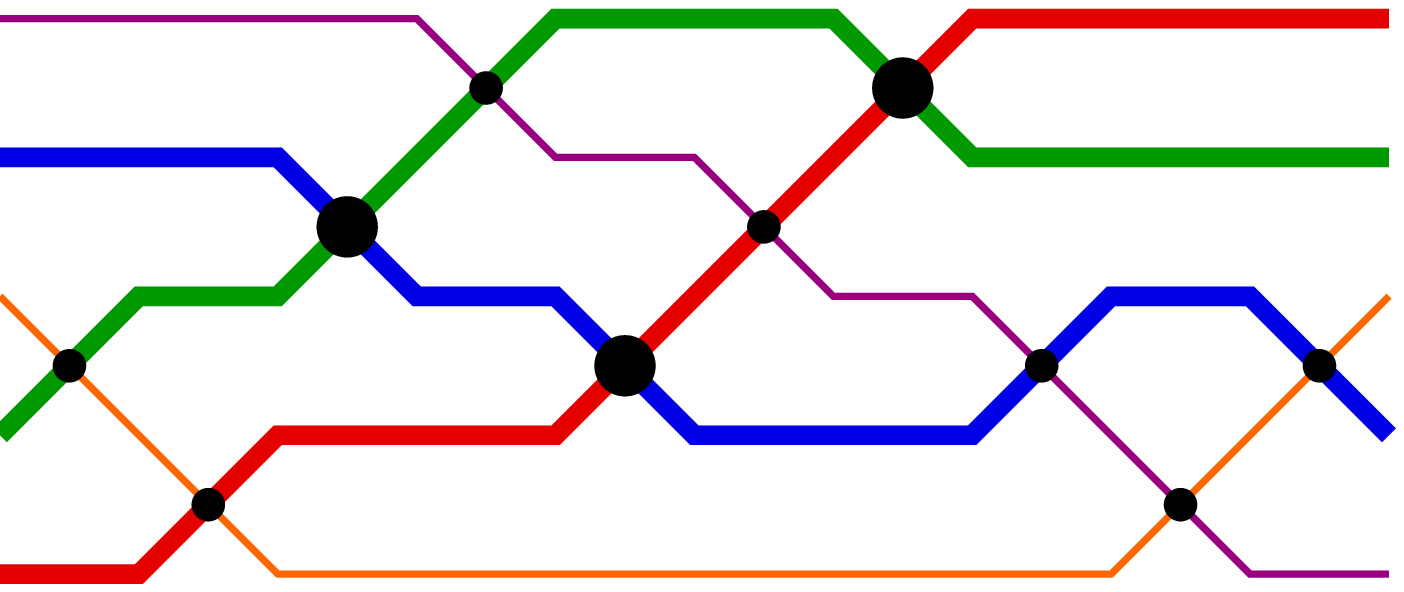}}
\hspace{2cm}
\resizebox{!}{3.1cm}{\includegraphics{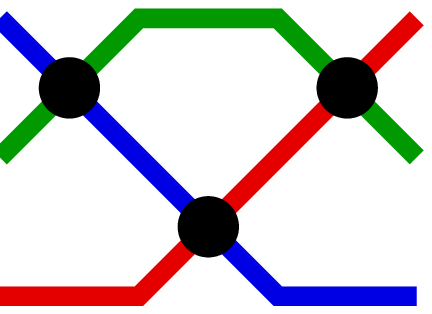}}
\caption{An illustration of the $5$-particle sorting network
  $\omega=(2,1,3,4,2,3,4,2,1,2)$, together with the 3-particle subnetwork
  $\omega|_A=(2,1,2)$ induced by the subset of particles $A=\{1,2,4\}$.}
\label{subnet}
\end{figure}

The {\dof uniform sorting network} $\omega_n$ is a random sorting network
chosen according to the uniform measure on the set all $n$-particle sorting
networks. For $m\leq n$, the {\dof random $m$-out-of-$n$ subnetwork}
$\omega^n_m$ is the random $m$-particle sorting network
$(\omega_n)|_{\mathcal A}$, where $\omega_n$ is a uniform $n$-particle
sorting network, and $\mathcal{A}$ is an independent uniformly random
$m$-element subset of $\{1,\ldots,n\}$.

Uniform sorting networks were investigated in \cite{ahrv}, leading to many
striking results and conjectures. (A different probability measure on
sorting networks was considered in \cite{oriented}.) Our main result is the
following surprising fact about random subnetworks. We denote the {\dof
  falling factorial} $(a)_r:=a(a-1)\cdots(a-r+1)$ (so $r!=(r)_r$).
\begin{thm}\label{main}
  Let $m\leq n$. In the random $m$-out-of-$n$ subnetwork, the expected
  number of swaps in location $j$ does not depend on $n$, and equals
  \[
  \E \#\big\{t: s_t(\omega^n_m)=j\big\} =
  \frac{(j-\h)_{j-1}(m-j-\h)_{m-j-1}}{(j-1)!(m-j-1)!},
  \qquad 1\leq j \leq m-1.
  \]
\end{thm}

Given only that the left side does not depend on $n$, the formula on the
right side may be recovered by reducing to the case $n=m$, which gives
$\omega^n_m=\omega_m$, and using known results on the uniform sorting
network (specifically, Proposition \ref{semi} below).

It is natural to seek generalizations of Theorem \ref{main}. For example
one might ask whether the law of $\#\big\{t: s_t(\omega^n_m)=j\big\}$ is
the same for each $n\geq m$. This is true for $m=3$ (indeed the law of
$\omega^n_3$ is the same for all $n\geq 3$), but fails for $m=4$ with
$n=4,5$.

From the case $m=4$ of Theorem \ref{main} we deduce the following result,
which was conjectured by Warrington \cite{warrington}. (We abbreviate
$(s_1,\ldots,s_N)$ to $s_1\cdots s_N$).

\begin{cor}\label{four}
  For all $n\geq 4$, the probability that the random $4$-out-of-$n$
  subnetwork lies in $\{123212,321232,212321,232123\}$ is $1/4$.
\end{cor}

The present work was triggered by Warrington's conjecture.
Corollary \ref{four} has a natural interpretation in terms of geometric
sorting networks, which we define next. Consider a set of $n$ points in
$\R^2$ with no three collinear and no two in the same vertical line. The
associated {\dof geometric sorting network} is defined as follows. Label
the points $x_1,\ldots,x_n$ in order of their projections onto the
horizontal axis. For all but finitely many angles $\theta$, the projections
onto the line through $0$ in direction $\theta$ fall in an order
$x_{\sigma(1)},\ldots,x_{\sigma(n)}$ corresponding to a permutation
$\sigma=\sigma_\theta$ of the original order, and as $\theta$ is increased
from $0$ to $\pi$, these permutations form the configurations of a sorting
network.

The four networks listed in Corollary \ref{four} are precisely those
geometric networks in which one point is in the convex hull of the other
three. It turns out that all $n$-particle sorting networks are geometric
for $n\leq 4$, but not for $n\geq 5$, as proved in \cite{goodman}. In fact
it is proved in \cite{pattern} that the uniform sorting network is
non-geometric with probability tending to $1$ as $n\to\infty$; on the other
hand, a principal conjecture of \cite{ahrv} is that in a certain sense the
uniform sorting network is approximately geometric.

The conjectures in \cite{ahrv} lead to the following precise prediction for
the limiting law of the random $m$-out-of-$n$ subnetwork as $n\to\infty$.
We will prove that Theorem \ref{main} is consistent with this conjecture.

\begin{conj}\label{geom}
  Let $X_1,\ldots,X_m$ be independent identically distributed random points
  in $\R^2$ chosen according to the {\em Archimedes} density
  \[
  \frac{1}{2\pi\sqrt{1-x^2-y^2}}
  \]
  on the disc $x^2+y^2< 1$, and let $\widehat\omega_m$ be the associated
  random geometric sorting network. The random $m$-out-of-$n$ subnetwork
  $\omega^n_m$ satisfies the convergence in distribution
  \[
  \omega^n_m\stackrel{D}{\to} \widehat\omega_m\qquad\text{as }n\to\infty.
  \]
\end{conj}

Conjecture \ref{geom} is implied by \cite[Conjecture 3]{ahrv} (and this is
implicit in the discussion at the end of \cite[Section 1]{ahrv} and
\cite[proof of Theorem~5]{ahrv}). The conjecture implies that any statistic
of the law of $\omega^n_m$ should converge to the appropriate
limit; we establish that this indeed holds for the expected value in
Theorem \ref{main}.

\begin{prop}\label{consistent}
  Let $\widehat\omega_m$ be the random geometric sorting network of
  Conjecture \ref{geom}. The expected number of swaps in $\widehat\omega_m$
  at location $j$ equals the right side in Theorem \ref{main}.
\end{prop}

\section{Proof of main result}

We will use the following key properties of uniform sorting networks.

\begin{prop}[\cite{ahrv}]\label{semi}
  Consider a uniform $n$-particle sorting network, and write
  $N=\binom{n}{2}$.
  \begin{mylist}
  \item The random sequence of swap locations is stationary. That is, \\
    $(s_1,\ldots,s_{N-1})$ and $(s_2,\ldots,s_N)$ are equal in law.
  \item The probability mass function $p_n$ of the first swap location is
    given by
    \[
    p_n(k)=\P(s_1=k)= \frac 1N \cdot
    \frac{(k-\h)_{k-1}(n-k-\h)_{n-k-1}}{(k-1)!(n-k-1)!},
    \quad 1\leq k\leq n-1.
    \]
  \end{mylist}
\end{prop}

Proposition \ref{semi} is proved in \cite[Theorem 1(i) and Proposition
9]{ahrv}.
For the reader's convenience we also summarize the arguments here. Part (i)
follows immediately because $(s_1,\ldots,s_N)\mapsto(s_2,\ldots,s_N,n-s_1)$
is a permutation of the set of all $n$-particle sorting networks. Part (ii)
requires more technology. By a bijection of Edelman and Greene
\cite{edelman-greene}, the location of the first swap is equal in law to
the position along the diagonal of the largest entry in a uniformly random
standard Young tableau of shape $(n-1,\ldots,2,1)$. The mass function of
the latter may be computed by using the hook formula of Frame, Robinson and
Thrall \cite{frame} to enumerate Young tableaux with and without a given
cell on the diagonal.

The stationarity of the uniform sorting network will play a key role in our
proof of Theorem \ref{main}.  We remark that the random $m$-out-of-$n$ subnetwork $\omega^n_m$ is {\em not} in general stationary; even $\omega^5_4$ is a counterexample, as noted in \cite{warrington}.  Stationarity apparently also fails for the random geometric sorting network $\widehat\omega_m$ of Conjecture \ref{geom} (according to simulations), and therefore it presumably fails to hold asymptotically for $\omega^n_m$ as $n\to\infty$.

Our proof of Theorem \ref{main} will proceed by relating the mass function
$p_n$ to a {\dof Polya urn} process, which is defined as follows. An urn
contains black and white balls in some numbers (which for our purposes need
not be integers). At each step, one new ball is added to the urn; if the
urn currently contains $w$ white and $b$ black balls, the next ball to be
added is white with probability $w/(b+w)$, otherwise black.

\begin{lemma}\label{polya}
  Consider a Polya urn that initially contains $1\h$ black and $1\h$ white
  balls.
  \begin{mylist}
  \item The random sequence of colors of added balls is exchangeable (i.e.\
    invariant in law under all permutations affecting finitely many
    elements).
  \item After $n-2$ balls have been added, the probability that $k-1$ of
    them are white equals $p_n(k)$.
  \end{mylist}
\end{lemma}

(Property (i) is well known, for arbitrary initial numbers of balls).

\begin{proof}
  The probability of adding $k-1$ white followed by $n-k-1$ black balls is
  \begin{equation}\label{balls}
    \frac{1\h}{3}\frac{2\h}{4}\cdots\frac{k-\h}{k+1}
    \times\frac{1\h}{k+2}\frac{2\h}{k+3}\cdots\frac{n-k-\h}{n}
    =\frac{2\;(k-\h)_{k-1}(n-k-\h)_{n-k-1}}{n!}.
  \end{equation}
  Moreover, the probability of adding $k-1$ white and $n-k-1$ black balls
  in any given order also equals \eqref{balls}, since we obtain the same
  denominators, and the numerators in a different order. This gives the
  claimed exchangeability. Therefore, the probability that $k-1$ of the
  first $n-2$ balls are white is \eqref{balls} multiplied by
  $\binom{n-2}{k-1}$, which equals the right side in Proposition \ref{semi}
  (ii).
\end{proof}

Let $h^n_{m,k}$ be the mass function of a hypergeometric distribution,
i.e., let $h^n_{m,k}(i)$ be the probability of obtaining $i$ white balls
when $m$ are chosen at random without replacement from an urn containing
$n$ balls of which $k$ are white. So
\[
h^n_{m,k}(i)=h^n_{k,m}(i) =
\frac{\binom{k}{i}\binom{n-k}{m-i}}{\binom{n}{m}}
\]
(where as usual we take $\binom{a}{b}=0$ if $b\notin[0,a]$).

\begin{lemma}\label{sum}
  For integers $n,m,j$ satisfying $m\leq n$ and $1\leq j\leq m-1$,
  \[ \sum_{k\in\Z} p_n(k) h^{n-2}_{m-2,k-1}(j-1)=p_m(j). \]
  In particular the left side does not depend on $n$.
\end{lemma}

\begin{proof}
  Consider the Polya urn of Lemma \ref{polya}. When $n-2$ balls have been
  added, suppose $m-2$ balls are chosen at random from these $n-2$. Then
  the left side is the probability that $j-1$ of those chosen are white. By
  the exchangeability in (i), this probability remains the same if we
  condition on the event that the chosen balls are the first $m-2$ to be
  added to the urn, but then the probability is clearly $p_m(j)$ by Lemma
  \ref{polya} (ii).
\end{proof}

We remark that a direct computational proof of Lemma \ref{sum} is also
possible, using induction on $n$.

\begin{proof}[Proof of Theorem \ref{main}]
  Consider the random uniform $m$-out-of-$n$ subnetwork
  $\omega^n_m=(\omega_n)|_\mathcal{A}$. Let $q(n,m,k,j,t)$ be the
  probability that the $t$th swap in the $n$-particle network $\omega_n$
  occurs in location $k$, and that this swap corresponds to some swap in
  location $j$ in the $m$-out-of-$n$-network. By the stationarity in
  Proposition \ref{semi}(i), and the fact that selecting $m$ items at
  random from $n$ is invariant under permutations of the $n$ items, $q$ is
  constant in $t$. On the other hand we have
  \[
  q(n,m,k,j,1)=p_n(k)\; h^n_{m,2}(2)\; h^{n-2}_{m-2,k-1}(j-1),
  \]
  since given that the first swap in $\omega_n$ has location $k$, the event
  in question occurs if and only if the $m$ chosen elements comprising
  $\mathcal A$ include the pair $k,k+1$, and exactly $j-1$ of
  $1,\ldots,k-1$. Now the required expectation is
  $\sum_{k\in\Z}\sum_{t=1}^N q(n,m,k,j,t)$. By the above observations,
  together with Lemma~\ref{sum} and the fact that
  $h^n_{m,2}(2)=\binom{m}{2}/\binom{n}{2}$, this sum equals
  $\binom{m}{2}p_m(j).$
\end{proof}

\section{Proofs of additional results}

\begin{proof}[Proof of Corollary \ref{four}]
  It is easy to check that of the $16$ $4$-particle sorting networks, the
  given $4$ each have $3$ swaps in location $2$, while the remaining $12$
  each have $2$. But Theorem \ref{main} gives that the expected number of
  swaps in location $2$ is $9/4=(1/4)3+(3/4)2$.
\end{proof}

\begin{proof}[Proof of Proposition \ref{consistent}]
  Let $X_1,\ldots,X_m$ be i.i.d.\ with Archimedes density as in
  Conjecture \ref{geom}. We start by noting two properties. First, the
  projection of $X_1$ onto any fixed direction has uniform distribution on
  $[-1,1]$. Second, the signed distance from $0$ of the line through $X_1$
  and $X_2$ has semicircle law, i.e.\ density function
  $\frac2\pi\sqrt{1-r^2}$ on $[-1,1]$. (See \cite[proof of
  Theorem~5]{ahrv}).

  Since each pair of particles swaps somewhere in $\widehat\omega_m$, it
  suffices to compute the probability that a given pair, say those
  corresponding to $X_1,X_2$, swap in location $j$ (and then multiply by
  $\binom{m}{2}$). This swap occurs when the rotating projection line is
  perpendicular to the line through $X_1$ and $X_2$, at which time the
  projections of $X_1$ and $X_2$ coincide, at a point $R$ with semicircle
  law. This swap is at location $j$ precisely if $j-1$ of $X_3,\ldots,X_m$
  are projected to the left (say) of $R$; but the projections of these
  points are uniform and independent of $R$. Thus the required expectation
  is
  \begin{equation}\label{ans}
    \binom{m}{2}\int_{-1}^1\binom{m-2}{j-1}
    \Big(\frac{1+r}{2}\Big)^{j-1}\Big(\frac{1-r}{2}\Big)^{m-j-1}
    \frac2\pi\sqrt{1-r^2}\;dr.
  \end{equation}
  Leaving aside multiplicative constants and applying the change of
  variable $t=(r+1)/2$, the integral reduces to a standard Beta integral
  (see e.g.\ \cite[p.~148]{hoel}):
  \[
  \int_0^1 t^{j-\frac12} (1-t)^{m-j-\frac12}\;dt
  = \frac{\Gamma(j+\h)\Gamma(m-j+\h)}{\Gamma(m+1)}.
  \]
  Using $\Gamma(\h)=\surd\pi$, a routine computation then shows that
  \eqref{ans} equals the right side in Theorem \ref{main}.
\end{proof}

We remark that the last computation may be viewed as an asymptotic version
of Lemma \ref{sum}, in the limit $n\to\infty$.

\pagebreak
\section*{Open questions}

\begin{mylist}
\item Does the law of the random $m$-out-of-$n$ sorting network
  converge as $n\to\infty$, for fixed $m$?  (Conjecture~\ref{geom} makes a specific prediction about the limit, but even its existence it is not known.)

\item Our use of the Polya urn can be viewed as a natural way to couple the
  law of the first swap location $s_1(\omega_n)$ for different values of
  $n$ -- indeed the coupling has the property that $s_1(\omega_{n+1}) -
  s_1(\omega_n) \in \{0,1\}$. Is there a natural way to couple the entire
  uniform sorting networks $\omega_n$ and $\omega_{n+1}$? For example, can
  it be done in such a way that $\omega_{n}=(\omega_{n+1})|_\mathcal{B}$,
  for some random $n$-element set $\mathcal{B}$?
\end{mylist}

\bibliographystyle{habbrv}
\bibliography{subnet}

\noindent {\sc Omer Angel:}
{\tt angel at math dot ubc dot ca}\\
{\sc Alexander E. Holroyd:}
{\tt holroyd at math dot ubc dot ca}\\
Department of Mathematics, University of British Columbia, \\
121--1984 Mathematics Road, Vancouver, BC V6T 1Z2, Canada.

\vspace{3mm} \noindent {\sc Alexander E. Holroyd:}\\
Microsoft Research, 1 Microsoft Way, Redmond, WA 98052, USA.

\end{document}